\documentclass[12pt]{amsart}

\usepackage{amsmath}
\usepackage{amssymb}
\usepackage{amsthm}
\usepackage{enumerate}
\usepackage{latexsym}
\usepackage{algorithmic}
\usepackage[ruled]{algorithm}
\usepackage{graphicx}
\usepackage{tikz}
\usepackage[colorlinks]{hyperref}

\usepackage[latin1]{inputenc}
\usepackage[english]{babel}
\usepackage{color}

\newtheorem{theorem}{Theorem}[section]
\newtheorem{corollary}[theorem]{Corollary}
\newtheorem{proposition}[theorem]{Proposition}
\newtheorem{lemma}[theorem]{Lemma}

\theoremstyle{definition}
\newtheorem{definition}[theorem]{Definition}
\newtheorem{example}[theorem]{Example}
\newtheorem{remark}[theorem]{Remark}

\newcommand{\reg}{\textnormal{reg}}
\newcommand{\lex}{\textnormal{lex}}
\newcommand{\sat}{\textnormal{sat}}

\newcommand{\Proj}{\textnormal{Proj}\,}

\begin{document}

\title[Regularities for ideals with a given Hilbert function]
{The range of all regularities for polynomial ideals with a given Hilbert function}

\author[]{Francesca Cioffi}
\address{Dip. di Matematica e Applicazioni \lq\lq R. Caccioppoli", Universit\`{a} degli Studi di Napoli Federico II \\  Via Cintia \\ 80126  Napoli \\ Italy}
\email{\href{mailto:cioffifr@unina.it}{cioffifr@unina.it}}
\thanks{The author is a member of \lq\lq National Group for Algebraic and Geometric Structures, and their Applications\rq\rq \ (GNSAGA - INDAM (Italy)).}


\begin{abstract}
Given the Hilbert function $u$ of a closed subscheme of a projective space over an infinite field $K$, let $m_u$ and $M_u$ be, respectively, the minimum and the maximum among all the Castelnuovo-Mumford regularities of schemes with Hilbert function~$u$. I show that, for every integer $m$ such that $m_u \leq m \leq M_u$, there exists a scheme with Hilbert function $u$ and Castelnuovo-Mumford regularity $m$. As a consequence, the analogous algebraic result for an O-sequence $f$ and homogeneous polynomial ideals over $K$ with Hilbert function $f$ holds too. 

Although this result does not need any explicit computation, I also describe how to compute a scheme with the above requested properties. Precisely, I give a method to construct a strongly stable ideal defining such a scheme. 
\end{abstract}

\maketitle


\section*{Introduction}

Investigations on Castelnuovo-Mumford regularity are still very active, due to the interesting consequences that new results in this topic have in order to better understand the behavior of projective schemes and, more generally, of graded algebras. Recently, a new contribution in this context has been given in  \cite{HoaArXiv} (see also the references therein) about bounds that are described by means of Hilbert coefficients.
  
In the present manuscript, I face the following question that has been posed to me by Le Tuan Hoa in a private communication: \lq\lq For a given Hilbert function $f$, the maximal Castelnuovo-Mumford regularity $M$ is attained by the lex-segment ideal $\lex(f)$ with Hilbert function $f$. On the other hand, there is a minimal value $m$ among all ideals with Hilbert function $f$. Given any $m< a < M$, can we find an ideal $I$ such that $\reg(I) = a$?\rq\rq

Here, Hilbert functions of finitely generated standard graded algebras over an infinite field $K$ are considered, hence Hilbert functions of standard graded algebras of type $K[x_0,\dots,x_n]/I$, where $I$ is a homogeneous polynomial ideal and the integer $n$ depends on the Hilbert function. In this setting, starting from a geometrical point of view in which a scheme will be understood to be a closed subscheme of a projective space, the first answer that is given is the following (see Theorem~\ref{th:main result}):
\begin{itemize}
\item[(A)] Let $m_u$ and $M_u$ be, respectively, the minimal and the maximal possible Castelnuovo-Mumford regularity for a scheme with Hilbert function $u$. For each integer $m$ such that $m_u\leq m \leq M_u$ there is a scheme $X$ with Hilbert function $u$ and $\reg(X)=m$.
\end{itemize}

Nevertheless, an analogous result immediately follows for the regularity of polynomial ideals with a given Hilbert function $f$, because there exists a polynomial ideal $J$ with Hilbert function $f$ and regularity $a$ if and only if there exists a scheme $X$ with Hilbert function $\sum f$ and Castelnuovo-Mumford regularity $a$ (see Corollary~\ref{cor:main cor}): 
\begin{itemize}
\item[(B)] Let $m$ and $M$ be, respectively, the minimal and the maximal possible regularity for a polynomial ideal $I$ such that $f$ is the Hilbert function of the finitely generated graded $K$-algebra $K[x_1,\dots,x_n]/I$. For every integer $a$ such that $m\leq a \leq M$ there is a {\em strongly stable ideal} $J$ so that $\reg(J)=a$ and $f$ is the Hilbert function of $K[x_1,\dots,x_n]/J$.
\end{itemize}

In \cite{CLMR2015} some special techniques have been developed in order to compute the minimal possible Castelnuovo-Mumford regularity $m_{p(z)}$ for a scheme with a given Hilbert polynomial $p(z)$. Among other results, also a closed formula has been given for the minimal possible Castelnuovo-Mumford regularity $m_u$ of a scheme $X$ with a given Hilbert function $u$ (\cite[Theorem~8.4]{CLMR2015} and Theorem \ref{th:minreg}). On the other hand, the maximal possible Castelnuovo-Mumford regularity $M_u$ of such a scheme $X$ coincides with the regularity of the lex-segment ideal with Hilbert function $\Delta u$ (Proposition \ref{prop: max reg}).

Although the results of \cite{CLMR2015} have been stated on a field of characteristic zero, here we can apply them over an infinite field $K$, using the notion of zero-generic initial ideal that has been introduced and studied in \cite{CS2015}. 

I exploit and improve some of the techniques that have been introduced in \cite{CLMR2015}, more precisely the notion of {\em minimal function} (see Definition~\ref{def:minimal functions}), and the construction that has been called an {\em expanded lifting} \cite[Theorem~7.4]{CLMR2015} and which here is re-stated over an infinite field (see Theorem \ref{th:construction}). 
The new insights that are presented (see Theorem~\ref{th:new}) together with some technical results give a proof for statements (A) and (B). In particular, it will be useful the fact that the Hilbert function of the unique saturated lex-segment ideal with a given admissible Hilbert polynomial $p(z)$ attains $p(z)$ exactly from the Gotzmann number of $p(z)$ minus $1$ (see Proposition~\ref{prop:CMS e altro}). 

Although the proof of statements (A) and (B) does not need any explicit computation, in Section \ref{sec:explicit construction} I show how to directly compute a strongly stable ideal with a given Hilbert function and regularity. To this aim, I employ the notion of growth-height-lexicographic Borel set, that has been introduced and used by D.~Mall in \cite{Mall2} to prove the connectedness of the stratum of schemes with the same Hilbert function in a Hilbert scheme over a field of characteristic zero. In \cite{Pardue97} K.~Pardue proves this connectedness over any infinite field.



\section{Preliminaries}

Let $K$ be an infinite field and $S:=K[x_0,\dots,x_n]=\bigoplus_{t\geq 0} S_t$ be the ring of polynomials over $K$ in $n+1$ variables with $x_0\prec x_1\prec \dots \prec x_n$, where $S_t$ is the $K$-vector space of the homogeneous polynomials of degree $t$. 
For a subset $N\subseteq S$ we set $N_t:=N\cap S_t$. 
A {\em term} of $S$ is a power product $x^{\alpha}:=x_0^{\alpha_0} x_1^{\alpha_1} \dots x_n^{\alpha_n}$ where ${\alpha_0},\dots,\alpha_n$ are non-negative integers, $\min(x^\alpha)$ denotes the least variable appearing in $x^\alpha$ with a positive exponent and $\mathbb T$ is the multiplicative monoid of all terms of $S$. 


Given a degree $t$, a set $B\subset \mathbb T_t$ is a {\em lex-segment} if it consists of the $\vert B\vert$ highest terms of $\mathbb T_t$ with respect to the deglex order. 
Given a subset $N\subset\mathbb T_t$, a {\em lex-segment in $N$} is the intersection of a lex-segment of $\mathbb T_t$ and $N$.  
A monomial ideal $J$ is a {\em lex-segment ideal} if $J_t$ is a lex-segment for every integer $t$. 
 
For a homogeneous ideal $I\subset S$, the Hilbert function of the graded algebra $S/I$ is denoted by $H_{S/I}$ and its Hilbert polynomial by $p_{S/I}(z)$.

Let $\Delta^0 H_{S/I}(t):=H_{S/I}(t)$ and, for every $i\geq 1$, define the function $\Delta^i H_{S/I}: \mathbb{N} \longrightarrow\mathbb{N}$ such that $\Delta^i H_{S/I}(0):=1$ and $\Delta^i H_{S/I}(t):= \Delta^{i-1}H_{S/I}(t)-\Delta^{i-1}H_{S/I}(t-1)$, for every $t>0$, calling it {\em i-th derivative} of $H$; we use an analogous notation for Hilbert polynomials. 

Also define the function $\Sigma H_{S/I}: \mathbb{N} \longrightarrow\mathbb{N}$ letting $\Sigma H_{S/I}(0):=1$ and $\Sigma H_{S/I}(t):= \Sigma H_{S/I}(t-1)+H_{S/I}(t)$ for each $t\geq 1$, and call it {\em integral} of $H_{S/I}(t)$; we use an analogous notation for Hilbert polynomials. 

Given a positive integer $t$, any positive integer $a$ can be uniquely written in the form 
\begin{equation}
\label{espansione}
\begin{array}{lllll}
a & = & \binom{k(t)}{t} + \binom{k(t-1)}{t-1} + \dots + \binom{k(j)}{j} + \dots + \binom{k(1)}{1} &=: &a_t,
\end{array}
\end{equation} 
where $k(t)> k(t-1)>\dots > k(1)\geq 0$ and with the convention that a binomial coefficient $\binom{n}{m}$ is null whenever either $n<m$ or $m<0$, and $\binom{n}{0}=1$, for all $n\geq 0$ (see \cite[Lemma~4.2.6]{BH}). If $j$ is the minimum integer $i$ such that $k(i)\geq i$, we simply write $a_t=\binom{k(t)}{t} + \binom{k(t-1)}{t-1} + \dots + \binom{k(j)}{j}$. Note that $j\geq 1$. The unique writing in formula \eqref{espansione} is called {\em binomial expansion of $a$ in base $t$} or {\em $t$-th Macaulay representation of $a$}. 
Referring to \cite{Ro}, we let 
$$\begin{array}{ll}
(a_t)^+_+&:=\binom{k(t)+1}{t+1} + \binom{k(t-1)+1}{t} + \dots + \binom{k(j)+1}{j+1},  \text{ and }\\
(a_t)^-_-&:=\binom{k(t)-1}{t-1} + \binom{k(t-1)-1}{t-2} + \dots + \binom{k(j)-1}{j-1}.
\end{array}$$
A numerical function $H:\mathbb N \rightarrow \mathbb N$ is an {\em O-sequence} if $H(0)=1$ and $H(t+1)\leq (H(t)_t)^+_+$. 

Thanks to \cite{Ma}, a numerical function $H$ is an O-sequence if, and only if, it is the Hilbert function of $S/I$ for a suitable homogeneous ideal $I$ if, and only if, there exists the (unique) lex-segment ideal $\lex(H)\subset S$ such that $H$ is the Hilbert function of $S/\lex(H)$. For any O-sequence $H$, we also denote by $p_H(z)$ its Hilbert polynomial. 

Given a term order $\sigma$ and a homogeneous ideal $I\subset S$, denote by $\mathrm{in}_\sigma(I)$ the initial ideal of $I$ with respect to $\sigma$ (w.r.t.~$\sigma$, for short). Then, $S/I$ and $S/\mathrm{in}_\sigma (I)$ have the same Hilbert function. Thanks to A.~Galligo in charactestic zero and to D.~Bayer and M.~Stillman in any characteristic, in generic coordinates the initial ideal of $I$ w.r.t.~$\sigma$ is a constant Borel-fixed ideal, which is called the {\em generic initial ideal} of $I$ and denoted by $\mathrm{gin}_\sigma(I)$. It preserves many of the properties of $I$, especially when the term order $\sigma$ is the degrevlex.

A monomial ideal $J$ is {\em strongly stable} if for every term $x_1^{\alpha_1}\dots x_n^{\alpha_n}$ in $J$ and variable $x_i$ with $\alpha_i>0$, the term $x_1^{\alpha_1}\dots x_n^{\alpha_n}\cdot \frac{x_j}{x_i}$ belongs to $J$, for every $x_j\succ x_i$.
A strongly stable ideal is Borel-fixed, and the converse holds in characteristic zero. A lex-segment ideal is strongly stable.

The {\em zero-generic initial ideal} of a homogeneous ideal $I\subset S$ w.r.t.~a term order $\sigma$ is the ideal $(\mathrm{gin}_\sigma((\mathrm{gin}_\sigma(I))_{\mathbb Q}))_{K}\subset K[x_1,\dots,x_n]$ (see \cite[Definition~2.1]{CS2015}), simpler denoted by $\mathrm{gin}_\sigma(\mathrm{gin}_\sigma(I)_{\mathbb Q})_{K}$. 
It is a strongly stable ideal, independently from the characteristic of the field \cite[Proposition~2.2(i)]{CS2015}.


\section{Regularities and some interactions}

Denoting by $S$ the ring of polynomials over $K$ in $n+1$ variables, $\mathbb P^n_K:=\Proj(S)$ is the $n$-dimensional projective space over $K$.

\begin{definition}\label{def:regfunction}
Given an O-sequence $u$ with Hilbert polynomial $p(z)$, the \emph{regularity of $u$} is 
$\varrho_u:=\min\{t\in \mathbb N \ \vert \ u(t')=p(t'), \forall \ t'\geq t\}$.
\end{definition}

\begin{definition}\label{def:regularity}
A homogeneous ideal $I\subset S$ is {\em $m$-regular} if the $i$-th syzygy module of $I$ is generated in degree $\leq m+i$. The {\em regularity} $\reg(I)$ of $I$ is the smallest integer $m$ for which $I$ is $m$-regular. 
\end{definition}

For a survey on the description of the regularity of a finitely generated graded $S$-module in terms of local cohomology see \cite{HoaLecture,RTV2003}. 

Recall that a homogeneous ideal $I\subset S$ is said {\em saturated} if it coincides with its {\em saturation} $I^{\sat}:=\{f \in S \ \vert \ \forall \ i\in {0,\dots,n}, \exists\ k_i : x_i^{k_i}f \in I\}$. 

\begin{definition}\label{def:CMregularity}
With the common notation of the ideal sheaf cohomology, given a scheme $X\subset \mathbb{P}^n_K$ and its (saturated) defining ideal $I=I(X)$, the {\em Castelnuovo-Mumford regularity} of $X$ is $\reg(X):=\min\{t\in \mathbb N\ \vert \ H^i(\mathcal I_X(t'-i))=0, \forall \ t'\geq t, \forall i > 0 \}$ and it is equal to $\reg(I)$. We set $H_X(t):=H_{S/I}(t)$, $p_X(z):=p_{{S}/{I}}(z)$, $\varrho_X:=\varrho_{S/I}$. 
\end{definition}

\begin{remark}\label{rem:zero-generic}
If $J\subset S$ is a strongly stable ideal, then $\reg(J)$ is the maximal degree of its minimal generators \cite[Proposition 2.9]{BS}. This is not true for a Borel-fixed ideal.
The zero-generic initial ideal of a homogeneous ideal $I\subset K[x_1,\dots,x_n]$, w.r.t.~the degrevlex term order, is a strongly stable ideal which preserves the property to be saturated, the Hilbert function and the regularity \cite[Prop. 2.1 and Rem. 3.2]{CS2015}. 
\end{remark}

It is well known that the first derivative of the Hilbert function of a scheme is an O-sequence, and the converse is true thanks to \cite[Corollary 3.4]{GMR}. 

More precisely, let $X\subset \mathbb P^n_K$ be a scheme of dimension $k>0$, with defining ideal $I=I(X)$,  $h\in S_1$ be a general linear form that is not a zero-divisor on $S/I$, and $J:=(I,h)$. Then, the scheme $Z\subset{\mathbb P}^{n-1}_K$ of dimension $k-1$ defined by the saturated ideal $(J/(h))^{\sat} = ((I,h)/(h))^{\sat}$ is a general hyperplane section of $X$. The short exact sequence
$0 \rightarrow (S/I)_{t-1} \xrightarrow{\cdot h} (S/I)_t \rightarrow (S/J)_t \rightarrow 0$ \ 
gives $H_{S/J}(t) = \Delta H_{S/I}(t)$, $p_{{S}/{(I,h)}}=\Delta p_{{S}/{I}}$, and then $\varrho_{S/J}=\varrho_X+1$, so that $\Delta H_X(t)=H_{S/J}(t)\geq H_Z(t)$, for every $t$, and $H_{S/J}(t)=H_Z(t)$, for every $t\geq \max\{\varrho_Z,\varrho_{S/J}\}$. 

\begin{lemma}\label{lemma: ex CMR newreg}
With the above notation,
\begin{itemize}
\item[(i)] $\reg(X)=\reg((I(X),h))$ \cite[Lemma (1.8)]{BS}.
\item[(ii)] $\reg(X)=\max\left\{\reg(Z),\varrho_X+1\right\}$ \cite[Lemma 3.6]{CMR}.
\end{itemize}
\end{lemma}

The quoted relation between the first derivative of the Hilbert function of $X$ and the Hilbert function of $Z$ suggests to consider the following partial order.

\begin{definition}\cite{R}\label{order} 
Given two sequences of integers $A=(a_i)_{i\in \mathbb N}$ and $B=(b_i)_{i\in \mathbb N}$, we let $A\preceq B$, if $a_i\leq b_i$ for every index $i$. 
\end{definition}

The polynomials $p(z)\in \mathbb Q[z]$ that are Hilbert polynomials of schemes are called {\em admissible} and are completely characterized in \cite{H66}. Given an admissible polynomial $p(z)$ with positive degree, its first derivative $\Delta p(z)$ is admissible too, because it is the Hilbert polynomial of a general hyperplane section of a scheme. 

The {\em Gotzmann number} $r_{p(z)}$ of an admissible polynomial $p(z)$ is the best upper bound for the Castelnuovo-Mumford regularity of a scheme having $p(z)$ as Hilbert polynomial (see \cite{Go}). 
Precisely, $r_{p(z)}$ is the regularity of the unique saturated lex-segment ideal defining a scheme with Hilbert polynomial $p(z)$. 
I refer to \cite{Gr} for an overview of these arguments and here only recall that if the degree of $p(z)$ is positive then $r_{\Delta p(z)} \leq r_{p(z)}$.

\begin{remark}\label{rem:varie lex}
Let $I\subseteq S$ be a homogeneous ideal, $u$ and $p(z)$ the Hilbert function and the Hilbert polynomial of $S/I$, respectively. Then, $\reg(I)\leq \reg(\lex(u))=\min\{j\geq \varrho_u \ \vert \ (p(t)_t)^+_+ = p(t+1), \forall t\geq j\}= \max\{\varrho_u, r_{p(z)}\}$ (see \cite{Bigatti, Hulett, Pardue97} and \cite[Theorem 2.5]{CMS}). 
\end{remark}

From now, I denote by $F(p(z),\varrho)$ the set of the Hilbert functions of schemes with regularity $\varrho$ and Hilbert polynomial $p(z)$. Moreover, for every $u\in F(p(z),\varrho)$, I denote by $m_u$ (resp.~$M_u$) the minimal (resp.~maximal) possible Castelnuovo-Mumford regularity of a scheme with Hilbert function $u$. 

The integer $m_u$ is computed and also characterized by a closed formula (see \cite[Theorem~8.4]{CLMR2015} and Theorem~\ref{th:minreg}). The integer $M_u$ can be characterized in the following way. 

\begin{proposition}\label{prop: max reg} {\rm (Maximal Castelnuovo-Mumford regularity for a scheme with a given Hilbert function)}
For every $u\in F(p(z),\varrho)$, $M_u=\reg(\lex(\Delta u))$.
\end{proposition}

\begin{proof}
$M_u\leq \reg(\lex(\Delta u))$ by Lemma \ref{lemma: ex CMR newreg}(i) and Remark \ref{rem:varie lex}. If $\lex(\Delta u)$ is contained in $K[x_1,\dots,x_n]$, then $\lex(\Delta u)\cdot S$ is a saturated strongly stable ideal in $S$ and defines a scheme $X$ with Hilbert function $u$ and $\reg(X)=\reg(\lex(\Delta u)\cdot S)=\reg(\lex(\Delta u))$ by Lemma \ref{lemma: ex CMR newreg}(i). 
\end{proof}

It is noteworthy that every generic initial ideal that is considered in the proofs of the paper \cite{CLMR2015} can be replaced by a zero-generic initial ideal, with the consequence that all the results of \cite{CLMR2015} hold over every infinite field. The authors of \cite{CS2015} have already highlighted that the results presented, in particular, in \cite[Theorem A]{CLMR2015} hold in any characteristic \cite[Remark 3.2 and Theorem 3.3]{CS2015}.

For the sake of completeness, I now re-state \cite[Theorem 7.4]{CLMR2015} over an infinite field and also highlight some further implications that will be crucial in our arguments.

\begin{theorem}{\rm (Expanded lifting over an infinite field)}\label{th:construction} 
Let $u$ belong to $F(p(z),\varrho)$. If $Z$ is a scheme with Hilbert polynomial $\Delta p(z)$ and  Hilbert function $g$ such that $g\preceq \Delta u$, then there is a scheme $X$ with $H_X=u$ and $\reg(X)=\max\{\reg(Z),\varrho+1\}$. Equivalently, there exists a saturated strongly stable ideal with Hilbert function $u$ and regularity $\max\{\reg(Z),\varrho+1\}$.
\end{theorem}

\begin{proof}
The first part of the statement is exactly that of \cite[Theorem 7.4]{CLMR2015}. In order to show that it  holds over an infinite field, let $\sigma$ be the degrevlex term order. Then, in the proof of \cite[Theorem 7.4]{CLMR2015} that has been given in the paper \cite{CLMR2015} we can replace the generic initial ideal $\mathrm{gin}_\sigma(I(Z))$ of the saturated ideal $I(Z)$, defining the scheme $Z$, by the zero-generic initial ideal of $I(Z)$. 
Recall that this ideal is a strongly stable saturated ideal with the same Hilbert function and regularity of $I(Z)$. 
The second affirmation holds because we can apply to the zero-generic initial ideal of $I(Z)$ exactly the same construction that is instead performed on $\mathrm{gin}_\sigma(I(Z))$ in the proof of \cite[Theorem 7.4]{CLMR2015}, being a zero-generic initial ideal strongly stable.
\end{proof}


\section{From minimal functions to the main result}\label{sec:minimalFunctions}

In this section, old and new properties of minimal functions are collected, and then they are used to prove the main result of this paper. 

Given a Hilbert polynomial $p(z)$, the set $\Pi_{p(z)}:=\{1 \leq t \leq r : (p(t+h)_{t+h})^+_+ \geq p(t+h+1)\geq 1, \forall \ h\geq 0\}$ is non-empty and the integer 
\begin{equation}\label{eq:rho}
\varrho_{p(z)}:=\left\{\begin{array}{ll}
0, & \text{ if }\min\ \Pi_{p(z)}=1 \text{ and } p(0)=1 \\
\min\ \Pi_{p(z)}, & \text{ otherwise}
\end{array}\right.
\end{equation}
is the minimal regularity of Hilbert functions with Hilbert polynomial $p(z)$ \cite[Proposition~4.2]{CLMR2015}.
Further, the integer 
\begin{equation}\label{eq:barrho}
\bar \varrho_{p(z)}:= \max\{\varrho_{p(z)},\varrho_{\Delta p(z)}-1\}
\end{equation}
is the minimal regularity of Hilbert functions of schemes with Hilbert polynomial $p(z)$, in particular $F(p(z),\bar\varrho_{p(z)})\not= \emptyset$ \cite[Theorem~4.9]{CLMR2015}.

Like it will be recalled in Proposition~\ref{prop:minimalfunction}, for every $\varrho\geq \varrho_{p(z)}$ there exists the minimum among the Hilbert functions with Hilbert polynomial $p(z)$ and regularity $\varrho$ with respect to the partial order $\preceq$ of Definition~\ref{order}. This minimal function is of the following type.

\begin{definition} \label{def:minimal functions}\cite[Definition 4.1]{CLMR2015}
For every $\varrho\geq \varrho_{p(z)}$, we let
$$f_{p(z)}^{\varrho}(t):=\left\{\begin{array}{ll}
p(t), & \text{ if } t\geq \varrho, \\
(f_{p(z)}^{\varrho}(t+1)_{t+1})^-_- , & \text{ otherwise;}
\end{array}\right.$$ 
$$g_{p(z)}^{\varrho}(t):=\left\{\begin{array}{ll} p(t), & \text{ if } t\geq \varrho, \\ p(t)+1, &\text{ if } t=\varrho-1, \\(g_{p(z)}^{\varrho}(t+1)_{t+1})^-_- , & \text{ otherwise.}\end{array}\right.$$
\end{definition}

\begin{proposition}\label{prop:minimalfunction}\cite[Propositions 4.2 and 4.5, Lemma 4.8(iii)]{CLMR2015}\\
\noindent For every integer $\varrho$ such that $\varrho_{p(z)}\leq \varrho\leq r_{p(z)}-1$,
\begin{itemize}
\item[(1)] $f_{p(z)}^{\varrho}$ is an O-sequence, $f_{p(z)}^\varrho\succeq f_{p(z)}^{\varrho+1}$ and $f_{p(z)}^{\varrho}$ is the minimal Hilbert function with Hilbert polynomial $p(z)$ and regularity $\leq \varrho$. In particular, $f_{p(z)}^{\varrho_{p(z)}}$ has regularity $\varrho_{p(z)}$. 
\item[(2)] $g_{p(z)}^{\varrho}$ is an O-sequence; if the regularity of $f^\varrho_{p(z)}$ is $\varrho'<\varrho$, then $g_{p(z)}^{\varrho}$ is the minimal O-sequence with regularity $\varrho$ and Hilbert polynomial $p(z)$, and $f^\varrho_{p(z)}\preceq g_{p(z)}^{\varrho}$.
\end{itemize}
Moreover, for every $\varrho\geq \bar \varrho_{p(z)}$, $\Delta f_{p(z)}^{\varrho}$ is an O-sequence. 
\end{proposition}

Recall that we are interested in Hilbert functions of schemes.

\begin{theorem}\label{th:first derivative}\cite[Theorem 4.9]{CLMR2015}
$f^{\bar\varrho_{p(z)}}_{p(z)}$ belongs to $F\big(p(z),\bar\varrho_{p(z)}\big)$ and, 
for all $\varrho>\bar\varrho_{p(z)}$, 
$$F(p(z),\varrho)\not= \emptyset \Leftrightarrow
\left\{\begin{array}{ll} f_{p(z)}^\varrho \in F(p(z),\varrho) &\text{(i.e. } f^\varrho_{p(z)} \text{ has regularity } \varrho) 
\\
g_{p(z)}^\varrho \in F(p(z),\varrho), &\text{ otherwise.} 
\end{array}\right.$$
\end{theorem}

\begin{theorem} \cite[Theorem 8.4]{CLMR2015} {\rm (Minimal Castelnuovo-Mumford regularity for a scheme with a given Hilbert function)} \label{th:minreg}
Let $u\in F(p(z),\varrho)$ and $k$ be the degree of $p(z)$. Consider $\tilde\varrho_u :=\min\{t\geq \bar\varrho_{\Delta p(z)} : f^t_{\Delta p(z)}\preceq \Delta u\}$ and $\tilde f:=f^{\tilde\varrho_u}_{\Delta p(z)}\in F(\Delta p(z),\tilde\varrho_u)$. 
If $k=0$ then $m_u=\varrho+1$, if $k>0$ then $m_u=\max\{m_{\tilde f}, \varrho+1\}$.
\end{theorem}

\begin{remark}\label{rem:varrho'}
If the regularity of $f_{p(z)}^{\varrho}$ is $\varrho'<\varrho$, then $f_{p(z)}^{\varrho'}=f_{p(z)}^{\varrho}$. However, a constant polynomial $p(z)=d$ has Gotzmann number $r_d=d$ and, for every $1\leq \varrho\leq d-1$, the minimal function $f_d^{\varrho}$ has regularity equal to $\varrho$ and belongs to $F(p(z),\varrho)$. Moreover, every scheme $X$ with Hilbert function $f_d^{\varrho}$ has $\reg(X)=\varrho+1$, because $X$ is Cohen-Macaulay. 
\end{remark} 

\begin{proposition}\label{prop:CMS e altro}
If $p(z)$ is a non-constant admissible polynomial, $r:=r_{p(z)}$ is its Gotzmann number and $r>r_{\Delta p(z)}$, then $f^{r-1}_{p(z)}$ belongs to $F(p(z),r-1)$.
\end{proposition}

\begin{proof}
The function $f_{p(z)}^{r-1}$ is exactly the Hilbert function of the unique saturated lex-segment ideal with Hilbert polynomial $p(z)$ (see \cite{Ro}), which has regularity $r$.
If $\varrho:=\varrho_{f^{r-1}_{p(z)}}$ denotes the regularity of $f^{r-1}_{p(z)}$, then $r=M_{f^{r-1}_{p(z)}}=\reg(\lex(\Delta f^{r-1}_{p(z)}))=\max\{\varrho+1,r_{\Delta p(z)}\}=\varrho+1$, because $r>r_{\Delta p(z)}$, from Remark \ref{rem:varie lex} and Proposition \ref{prop: max reg}.
\end{proof}

Next theorem implies a refinement of Theorem~\ref{th:first derivative} and highlights that, for every $\varrho\geq \max\{r_{\Delta p(z)}, \varrho_{p(z)}\}$, the minimal function with regularity $\varrho$ is of type $f^{\varrho}_{p(z)}$. The following technical lemma will be useful.

\begin{lemma}\label{lemma:crucial1}
Let $p(z)$ be a non-constant Hilbert polynomial. If $\bar t$ is a positive integer such that $(p(\bar t)_{\bar t})^-_- = p(\bar t-1)$ and 
$(\Delta p(\bar t)_{\bar t})^+_+ = \Delta p(\bar t+1)$, then $(p(\bar t)_{\bar t})^+_+=p(\bar t+1)$ and $(p(\bar t+1)^-_-)=p(\bar t)$.
\end{lemma}

\begin{proof} Consider the two binomial expansions 
$p(\bar t)_{\bar t} = \binom{k(\bar t)}{\bar t}+\binom{k(\bar t-1)}{\bar t-1}+\dots+\binom{k(\bar j)}{\bar j}$ and $p(\bar t+1)_{\bar t+1} = \binom{h(\bar t+1)}{\bar t+1}+\binom{h(\bar t)}{\bar t}+\dots+\binom{h(j)}{j}$. From the hypothesis on $p(\bar t-1)$ we have:
\begin{center}
$0\not=\Delta p(\bar t)= \binom{k(\bar t)-1}{\bar t}+\binom{k(\bar t-1)-1}{\bar t-1}+\dots+\binom{k(\bar j)-1}{\bar j}$,
\end{center}
which gives the binomial expansion of $\Delta p(\bar t)$ in base $\bar t$ (see \eqref{espansione}). Thus
\begin{center}
$(\Delta p(\bar t)_{\bar t})^+_+= \binom{k(\bar t)}{\bar t+1}+\binom{k(\bar t-1)}{\bar t}+\dots+\binom{k(\bar j)}{\bar j+1}$
\end{center}
and if $(\Delta p(\bar t)_{\bar t})^+_+ = \Delta p(\bar t+1)$ then from $p(\bar t+1)=\Delta p(\bar t+1)+p(\bar t)$ we obtain
\begin{center}
$(p(\bar t)_{\bar t})^+_+=\binom{k(\bar t)+1}{\bar t+1}+\binom{k(\bar t-1)+1}{\bar t}+\dots+\binom{k(\bar j)+1}{\bar j+1}=
\binom{h(\bar t+1)}{\bar t+1}+\binom{h(\bar t)}{\bar t}+\dots+\binom{h(j)}{j}=p(\bar t+1)$
\end{center}
which are both the binomial expansion of a same integer in base $\bar t+1$, hence $j=\bar j+1$ and $h(i+1) = k(i)+1$, for every $\bar j \leq i \leq \bar t$. So, we also obtain $(p(\bar t+1)_{\bar t+1})^-_- = p(\bar t)$.
\end{proof}

\begin{theorem}\label{th:new}
Let $p(z)$ an admissible polynomial of positive degree.
For every integer $\varrho$ such that $r_{p(z)}>\varrho \geq \max\{r_{\Delta p(z)}, \varrho_{p(z)}\}$,  
\begin{itemize}
\item[(i)] the minimal Hilbert function $f^{\varrho}_{p(z)}$ has regularity $\varrho$;
\item[(ii)] 
$F(p(z),\varrho)$ is non-empty and every scheme with Hilbert function $u\in F(p(z),\varrho)$ has Castelnuovo-Mumford regularity $\varrho+1$. 
\end{itemize}
\end{theorem}

\begin{proof}
First, observe that $\bar\varrho_{p(z)}\leq \max\{r_{\Delta p(z)}, \varrho_{p(z)}\} \leq \varrho$, thanks to \eqref{eq:barrho}.

For what concerns item (i), if $\varrho=r_{p(z)}-1$ then the thesis follows from Proposition \ref{prop:CMS e altro}. Otherwise, I show that, if the minimal function $f^{\varrho}_{p(z)}$ has regularity $< \varrho$, then $\varrho < r_{\Delta p(z)}$. Let $r:=r_{p(z)}$.

If $f^{\varrho}_{p(z)}$ has regularity $< \varrho$, then $(p(\varrho)_\varrho)^-_- = p(\varrho-1)$ thanks to Definition \ref{def:minimal functions}. From Proposition \ref{prop:CMS e altro}, the inequality $p(r) \not= (p(r-1)_{r-1})^+_+$ holds. Thus, there exists a positive integer $\bar t\geq \varrho$ such that $(p(\bar t)_{\bar t})^-_- = p(\bar t-1)$ and $(p(\bar t+1)_{\bar t+1})^-_- \not= p(\bar t)$. So, from Lemma \ref{lemma:crucial1} it follows that $(\Delta p(\bar t)_{\bar t})^+_+ \not= \Delta p(\bar t+1)$, hence $\varrho\leq \bar t < r_{\Delta p(z)}$ by Remark \ref{rem:varie lex}. 

For what concerns item (ii), $F(p(z),\varrho)$ is non-empty thanks to item~(i). For the second assertion, note that, from Lemma \ref{lemma: ex CMR newreg}(ii), Proposition~\ref{prop: max reg} and Remark \ref{rem:varie lex}, we have  
$\varrho+1 \leq m_u \leq M_u=\reg(\lex(\Delta u))=\max\{\varrho+1, r_{\Delta p(z)}\}$. Thus, we can conclude because $\varrho\geq r_{\Delta p(z)}$. 
\end{proof}

\begin{example}
For the Hilbert polynomial $p(z)=5z-3$, the Gotzmann number is $r=7$, $\varrho_{p(z)}=\bar\varrho_{p(z)}=3$ and $\bar\varrho_{\Delta p(z)}=1$, $r_{\Delta p(z)}=5$ \cite[Example~4.4(3)]{CLMR2015}. The minimal possible Castelnuovo-Mumford regularity of a scheme with Hilbert polynomial $p(z)$ is $m_{p(z)}=5$ thanks to \cite[Theorem 8.1]{CLMR2015}. Moreover, $f^4_{p(z)}=f^3_{p(z)}$ and $f^5_{p(z)}\not= f^4_{p(z)}$ because $(p(4)_4)^-_-=p(3)$ but $(p(5)_5)^-_- < p(4)$, in fact $6=(\Delta p(4)_4)^+_+ > \Delta p(5)=5$. In terms of Lemma \ref{lemma:crucial1}, we have $\bar t=4 < r_{\Delta p(z)}$.
\end{example}

All the results that have been obtained until now finally converge in the proof of the following main theorem.

\begin{theorem}\label{th:main result} {\rm(Main Result)}
Let $u$ belong to $F(p(z),\varrho)$. There is a scheme $X$ with Hilbert function $u$ and $\reg(X)=m$ if and only if $m_u\leq m \leq \reg(\lex(\Delta u))$.
\end{theorem}

\begin{proof}
Let $k$ be the degree of $p(z)$.

Recall that if a scheme $X$ of dimension $k$ is Cohen-Macaulay then $\reg(X)=\varrho_X+k+1$, and a zero-dimensional scheme is always Cohen-Macaulay. So, if $k=0$, then $\reg(X)=\varrho_u+1=m_u=M_u=\reg(\lex(\Delta u))$ by Proposition \ref{prop: max reg}. 

Let $k>0$. 
If $r_{\Delta p(z)} < \reg(\lex(\Delta u))$, then $M_u=\reg(\lex(\Delta u))=\varrho_{\Delta u}=\varrho+1=m_u$ thanks to Proposition \ref{prop: max reg} and Remark \ref{rem:varie lex}.
So, we can focus on the case $r_{\Delta p(z)} = \reg(\lex(\Delta u))$ and  $m_u < \reg(\lex(\Delta u))= r_{\Delta p(z)}$. If $k=1$ and $p(z)=dz+c$, then for every integer $\varrho'$ such that 
$$\varrho+1\leq m_u\leq \varrho' < \reg(\lex(\Delta u))=r_d=d,$$ 
the minimal function $f^{\varrho'}_d$ has regularity $\varrho'$ and there exists a zero-dimensional scheme $Z$ with Hilbert function $f^{\varrho'}_d$ and $\reg(Z)=\varrho'+1$ (see Remark \ref{rem:varrho'}). Moreover, $f^{\varrho'}_d\preceq \Delta u$ due to Proposition \ref{prop:minimalfunction}(1). Then, from Theorem \ref{th:construction} we obtain a scheme $X$ with Hilbert function $u$ and $\reg(X)=\reg(Z)=\varrho'+1$. 

Assume now $k>1$ and, referring to Theorem \ref{th:minreg}, consider $\tilde\varrho_u :=\min\{t\geq \bar\varrho_{\Delta p(z)} : f^t_{\Delta p(z)}\preceq \Delta u\}$. By the definition of $\tilde\varrho_u$, the Hilbert function $\tilde f:=f^{\tilde\varrho_u}_{\Delta p(z)}$ belongs to $F(\Delta p(z),\tilde\varrho_u)$. Moreover, from Theorem~\ref{th:minreg}, we have $m_u=\max\{m_{\tilde f}, \varrho+1\}$.

If $m_u< \reg(\lex(\Delta \tilde f))$, then $\reg(\lex(\Delta \tilde f))=\max\{\tilde\varrho_u+1,r_{\Delta^2 p(z)}\}=r_{\Delta^2 p(z)}$ because  $\tilde\varrho_u+1 \leq m_{\tilde f}\leq m_u$. Then, by the inductive hypothesis, for every integer $m$ such that 
$$m_{\tilde f} \leq m_u \leq m \leq \reg(\lex(\Delta \tilde f))=r_{\Delta^2 p(z)},$$ 
there exists a scheme $Z$ with Hilbert function $\tilde f$ and $\reg(Z)=m$.
Hence, by Theorem~\ref{th:construction}, there exists a scheme $X$ with Hilbert function $u$ and $\reg(X)=m$, because $\tilde f \preceq \Delta u$ by definition. 
If $r_{\Delta^2 p(z)}=r_{\Delta p(z)}$ then we can conclude here.
Otherwise, we now consider $m$ such that 
$$\max\{m_u,r_{\Delta^2 p(z)}\} < m \leq r_{\Delta p(z)}.$$
In this range Theorem \ref{th:new}(i) can be applied because $m_u\geq m_{\tilde f}>\bar\varrho_{\Delta p(z)}\geq \varrho_{\Delta p(z)}$. Then, the Hilbert function $f^{m-1}_{\Delta p(z)}$ has regularity ${m-1}$ and, hence, there exists a scheme $Z$ with $\reg(Z)=m$ and Hilbert function $f^{m-1}_{\Delta p(z)} \preceq f^\varrho_{\Delta p(z)}\preceq \Delta u$. Due to Theorem \ref{th:construction}, we again have a scheme $X$ with Hilbert function $u$ and $\reg(X)=m$.
\end{proof}

\begin{corollary}\label{cor:main cor} {\rm (Algebraic Case)}
Let $f$ be the Hilbert function of a finitely generated graded $K$-algebra $K[x_1,\dots,x_n]/I$. 
\begin{itemize}
\item[(i)] $m_{\sum f}$ and $\reg(\lex(f))$ are, respectively, the minimal and the maximal possible regularities for the polynomial ideal $I$. 
\item[(ii)] For every integer $a$ such that $m_{\sum f}\leq a \leq \reg(\lex(f))$, there is a {\em strongly stable ideal} $J$ with $\reg(J)=a$ and $f$ the Hilbert function of $K[x_1,\dots,x_n]/J$.
\end{itemize}
\end{corollary}

\begin{proof}
It is enough to observe that, given an O-sequence $f$ with regularity $\varrho$, there exists a polynomial ideal $I$ with Hilbert function $f$ and $\reg(I)=a$ if and only there exists a scheme $X$ with Hilbert function $\sum f$ and $reg(X)=a$, that is $m_{\sum f}\leq a \leq \reg(\lex(f))$. Indeed, $u:=\sum f$ is the Hilbert function of a scheme and then one can apply Theorem \ref{th:main result} and Lemma \ref{lemma: ex CMR newreg}(i). Finally, denoting by $N\subset S$ the (saturated) zero-generic initial ideal of $I(X)$ w.r.t.~degrevlex order, we can take the strongly stable ideal $J:=(N,x_0)/(x_0)\subset K[x_1,\dots,x_n]$.
\end{proof}

\begin{example}
The O-sequence $f=(1,3,5,4,5,4)$ has Hilbert polynomial $p(z)=4$ and $\varrho_u=5 > r_{p(z)}$. The integral of $f$ is the function $\sum f=(1,4,10,14,4z+3)$ in  $ F(4z+3,4)$ and $m_{\sum f}=5$. On the other hand, $M_u=\reg(\lex(f))=5=m_u$, so every homogeneous polynomial ideal with Hilbert function $f$ has regularity $5$. Moreover, every scheme $X$ with Hilbert function $\sum f$ has $\reg(X)=5$.
\end{example}

\begin{example}
For the O-sequence $u=(1,4,10,20,35,55,80,28z-90)$ we obtain $m_u=9=\varrho_u+2$ and $\reg(\lex(\Delta(u))=28$. From Remark~\ref{rem:varrho'}, for every $9 < m < 28$, there is a scheme $Z$ with Hilbert function $f^{m-1}_{28}\preceq \Delta(u)$ and $\reg(Z)=m$. By Theorem \ref{th:construction} there is also a scheme $X$ with Hilbert function $u$ and $\reg(X)=\max\{\varrho_u+1,m\}=m$.

Analogously, if we consider $f=\Delta u=(1,3,6,10,15,20,25,26,28)$, for every $m\in [m_{\sum f}, \reg(\lex(f))]=[9,28]$, there exists a (homogeneous) polynomial ideal $I$ with Hilbert function $f$ and $\reg(I)=m$. 
\end{example}

\begin{example}
Take the Hilbert function $u=(1,5,2z^2+3z+1)$ of the Veronese surface $X\subseteq \mathbb P^4_k$, which has Castelnuovo-Mumford regularity $\reg(X)=3=m_u$. We have $\Delta u=(1,4,10,4z+1)$ and $M_u=\reg(\lex(\Delta u))=r_{4z+1}=7$. Moreover, $f^3_{4z+1}=(1,4,8,4z+1)$ has regularity $3$, $f^4_{4z+1}=(1,4,8,12,4z+1)$ has regularity $4$ and $f^5_{4z+1}=(1,3,6,10,15,4z+1)$ has regularity $5$, according to Theorem \ref{th:new}(i). 
\end{example}


\section{Explicit construction of a scheme with given Hilbert function and Castelnuovo-Mumford regularity}
\label{sec:explicit construction}

Although the proof of the main result of this paper does not need any explicit construction, in this section I describe a method for identifying and directly computing schemes that satisfy Theorem \ref{th:main result}, by means of the notion of growth-height-lexicographic Borel set, which has been introduced by D. Mall in \cite{Mall1,Mall2}. 

A set $B\subset \mathbb T_t$ is called a {\em Borel set} if, for every term $x^{\alpha}$ in $B$, $x^{\alpha}\cdot \frac{x_j}{x_i}$ belongs to $B$ for every $x_i$ for which $x^\alpha$ is divisible and $x_j\succ x_i$. Hence, a monomial ideal $J$ is strongly stable if $J_t$ is a Borel set, for every $t$.

\begin{remark}
In the setting of this paper, which considers an infinite field of any characteristic, the use of the above expression \lq\lq Borel set" is not entirely correct because the property to be strongly stable is not equivalent to the property to be Borel-fixed. However, with this warning, I choose to adopt the language already used in literature in characteristic zero by the authors whom I refer to.
\end{remark} 

\begin{definition} \cite[Def. 2.7 and 2.13]{Mall1}
Let $B\subset \mathbb T_t$ be a Borel set.\\ 
(i) The set $B^{(i)}=\{x^\alpha\in B : \min(x^\alpha)=x_i\}$ is the {\em growth class} of $B$ of {\em growth} $i$.
The {\em growth-vector} of $B$ is $gv(B):=(\vert B^{(0)}\vert,\vert B^{(1)}\vert\dots,\vert B^{(n)}\vert)$.\\
(ii) For all $i\in \{0,\dots,t\}$, let $B(i):=\{x_0^{\alpha_0} x_1^{\alpha_1}\dots x_n^{\alpha_n}\in B: \alpha_0=i\}$.
The {\em height-vector} of $B$ is the vector $hv(B):=(\vert B(0)\vert,\vert B(1)\vert,\dots,$ $\vert B(t)\vert)$.
\end{definition}

\begin{remark}\label{rem:shape}
Given a Borel set $B\subset\mathbb T_t$, we have:  for every $i\neq j$, $B^{(i)}\cap B^{(j)}=\emptyset$  and  $B{(i)}\cap B{(j)}=\emptyset$; $B(0) =\bigcup_{i\geq 1}B^{(i)}$; $B^{(0)} =\bigcup_{i\geq 1}B{(i)}$.
\end{remark}

\begin{proposition} \cite[Proposition 3.2]{Mall1} \label{prop:height-vector}
Let $B\subset\mathbb T_t$ be a Borel set and $I=(B)^{\sat}\subset S$ the saturation of the ideal generated by $B$. 
If $u$ is the Hilbert function of $S/I$, then:
\begin{itemize}
\item[(i)]\label{it:height-vector_i} $u(j)=\binom{j+n}{n}-\sum_{i=t-j}^t \vert B(i)\vert$, for every $j\leq t$;
\item[(ii)]\label{it:height-vector_ii} $u(j)-u(j-1)=\binom{j+n-1}{n-1}-\vert B(t-j)\vert $  for every $j\leq t$;
\item[(iii)]\label{it:height-vector_iii} $u(t+k)=\binom{t+k+n}{n}-\sum \binom{i+k}{k} \vert B^{(i)}\vert$, for every $k\geq 1$.
\end{itemize} 
\end{proposition}

\begin{corollary}\label{cor:vectors}
Let $J\subseteq S$ be a saturated strongly stable ideal with $reg(J)=m$. Let $p(z)$ and $u$ be the Hilbert polynomial and the Hilbert function of $S/J$, respectively. Then,
\begin{itemize}
\item[(i)] the height-vector of $J_m$ depends only on $u$ and $m$;
\item[(ii)] the growth-vector of $J_m$ depends only on $p(z)$ and $m$, and is the unique solution of the linear equations $\sum \binom{i+k}{k} x_i = \binom{m+k+n}{n}-u(m+k)$, $0\leq k\leq n$.
\end{itemize}
\end{corollary}

\begin{proof}
Both statements follow from Proposition \ref{prop:height-vector} with $t=m$. 
\end{proof}

Let me denote by $hv(u,m)$ and $gv(p(z),m)$ the height-vector and the growth-vector that are uniquely determined in Corollary \ref{cor:vectors} by an integer $m$ and  a Hilbert function $u$ with Hilbert polynomial $p(z)$.


\begin{definition}[{\cite[Definitions 2.10 and 2.14]{Mall1}, \cite[Definition 4.3]{Mall2}}]
\label{DefMall} 
\ 

(1) \label{it:DefMall_i} A Borel set $B\subset\mathbb T_t$ is {\em growth-lexicographic} if all growth classes of $B$ are lex-segments in the corresponding growth classes $\mathbb T_t^{(i)}$ of $\mathbb T_t$.

(2) \label{it:DefMall_ii} A Borel set $B\subset\mathbb T_t$ is {\em growth-height-lexicographic} if $B(0)$ is growth-lexicographic and $B(i)$ is a lex-segment in $\mathbb T_t(i)$ for all $0<i\leq t$. 
\end{definition}

Given a Borel set $B\subset\mathbb T_t$, let $L^{(i)}$ be the set of the $\vert B^{(i)}\vert$ greatest terms $x^\alpha$ with $\min(x^\alpha)=x_i$ w.r.t.~deglex order for every $i=1,\dots,n$, and $L(i)$ the set of the $\vert B(i)\vert$ highest terms w.r.t.~deglex order in $\mathbb T_m(i)=\{x^\alpha : \alpha_0=i\}$, for every $i=1,\dots,t$. Then, consider the set $L:=(\cup_{i=1,\dots,n} L^{(i)}) \bigcup (\cup_{i=1,\dots,t} L(i))$ (see also Remark \ref{rem:shape}).

\begin{theorem}\label{th:crucial} \cite[Theorem 5.7]{CLMR2015}
With the above notation, $L$ is the unique growth-height-lexicographic Borel set in $\mathbb T_t$ with $gv(L(B))=gv(B)$ and $hv(L)=hv(B)$.
Moreover, the ideals $J:=(B)^{\sat}$ and $I:=(L)^{\sat}$ have the same Hilbert function and $\reg(J)\leq \reg(I) \leq t$.
\end{theorem}

\begin{theorem}\label{th:strongly stable ideal}
Given a Hilbert function $u\in F(p(z),\varrho)$ and an integer $m$ such that $m_u\leq m \leq M_u$, let $L\subset \mathbb T_m$ be the growth-height-lexicographic Borel set with height-vector $hv(u,m)$ and growth-vector $gv(p(z),m)$. Then, the strongly stable ideal $I=(L)^{sat}$ defines a scheme with Hilbert function $u$ and Castelnuovo-Mumford regularity $m$.
\end{theorem}

\begin{proof}
From Theorem \ref{th:main result}, a scheme $X$ with Hilbert function $u$ and $\reg(X)=m$ does exist. Let $J$ be the zero-generic initial ideal of the saturated defining ideal $I(X)$ of $X$, which is a strongly stable saturated ideal with the same Hilbert function and same regularity of $I(X)$. In particular, $J_m \subseteq \mathbb T_m$ is a Borel set. Then, we can take the height-vector $hv(J_m)=hv(u,m)$ and the growth-vector $gv(J_m)=gv(p(z),m)$, as described in the proof of Corollary \ref{cor:vectors}. 
By Theorem~\ref{th:crucial}, there is the unique growth-height-lexicographic Borel set $L=(\cup_{i=1,\dots,n} L^{(i)}) \bigcup (\cup_{i=1,\dots,m} L(i))$, which is determined by $hv(u,m)$ and $gv(p(z),m)$. Moreover, the ideal~$(L)^{sat}$ defines a scheme with Hilbert function $u$ and Castelnuovo-Mumford regularity $m$.
\end{proof}

\begin{example}
Consider the Hilbert function $u=(1,4z)$ which has $m_u=3$ and $M_u=4$. Applying Proposition \ref{prop:height-vector} and Corollary \ref{cor:vectors} to $u$ and $m=3$ we find $hv(u,3)=(6,2,0,0)$ and $gv(4z,3)=(2,2,3,1)$. The saturation of the corresponding growth-height-lexicographic Borel set is the ideal $J=(x_3^2,x_2x_3,x_2^3)$ in $S=K[x_0,x_1,x_2,x_3]$.
For $m=4$, we find $hv(u,4)=(11,6,2,0,0)$ and $gv(4z,4)=(8,6,4,1)$. In this case, the saturation of the corresponding growth-height-lexicographic Borel set is the ideal $J=(x_3^2,x_3x_2,$ $x_3x_1^2,x_2^4)\subset S$. Observe that $J$ is the monomial ideal generated in $S$ by the lex-segment ideal in $K[x_1,x_2,x_3]$ with Hilbert function $\Delta u$, according to Proposition \ref{prop: max reg}.
\end{example}

I highlight that here a growth-height-lexicographic Borel set is computed only starting from a Hilbert function $u\in F(p(z),\varrho)$ and an integer $m$ such that $m_u\leq m \leq M_u$. Indeed, Theorem \ref{th:main result} guarantees that the growth-height-lexicographic Borel set $L$ in $\mathbb T_m$ determined by $hv(u,m)$ and $gv(p(z),m)$ gives $\reg((L)^{sat})=m$. An implementation in CoCoA-4.7.5 (see \cite{CoCoA}) of an algorithm that constructs $(L)^{sat}$ is available at http://wpage.unina.it/cioffifr/GrowthHeight.

\section*{Acknowledgments}

I am very grateful to Le Tuan Hoa for the kind suggestion of the problem that I face here and for the  feedback he gave me on this work.



\begin{thebibliography}{10}

\bibitem{CoCoA}
J.~Abbott, A.~M. Bigatti, and L.~Robbiano, \emph{{CoCoA}: a system for doing
  {C}omputations in {C}ommutative {A}lgebra.}, Available at
  \texttt{http://cocoa.dima.unige.it}, 2009.

\bibitem{BS}
David Bayer and Michael Stillman, \emph{A criterion for detecting
  {$m$}-regularity}, Invent. Math. \textbf{87} (1987), no.~1, 1--11.

\bibitem{Bigatti}
Anna~Maria Bigatti, \emph{Upper bounds for the {B}etti numbers of a given
  {H}ilbert function}, Comm. Algebra \textbf{21} (1993), no.~7, 2317--2334.

\bibitem{BH}
Winfried Bruns and J{\"u}rgen Herzog, \emph{Cohen-{M}acaulay rings}, Cambridge
  Studies in Advanced Mathematics, vol.~39, Cambridge University Press,
  Cambridge, 1993.

\bibitem{CS2015}
Giulio Caviglia and Enrico Sbarra, \emph{Zero-generic initial ideals},
  Manuscripta Math. \textbf{148} (2015), no.~3-4, 507--520.

\bibitem{CMS}
Marc Chardin and Guillermo Moreno-Soc\'{i}as, \emph{Regularity of lex-segment
  ideals: some closed formulas and applications}, Proc. Amer. Math. Soc.
  \textbf{131} (2003), no.~4, 1093--1102.

\bibitem{CLMR2015}
F.~Cioffi, P.~Lella, M.~G. Marinari, and M.~Roggero, \emph{Minimal
  {C}astelnuovo-{M}umford regularity for a given {H}ilbert polynomial}, Exp.
  Math. \textbf{24} (2015), no.~4, 424--437.

\bibitem{CMR}
F.~Cioffi, M.~G. Marinari, and L.~Ramella, \emph{Regularity bounds by minimal
  generators and {H}ilbert function}, Collect. Math. \textbf{60} (2009), no.~1,
  89--100.

\bibitem{HoaArXiv}
Le~Xuan Dung and Le~Tuan Hoa, \emph{A note on {C}astelnuovo-{M}umford
  regularity and {H}ilbert coefficients}, J. Algebra and its applications
  (2019), doi:10.1142/S02194988195019134. Available at
  http://arxiv.org/pdf/1809.07461.pdf.

\bibitem{GMR}
A.~V. Geramita, P.~Maroscia, and L.~G. Roberts, \emph{The {H}ilbert function of
  a reduced {$k$}-algebra}, J. London Math. Soc. (2) \textbf{28} (1983), no.~3,
  443--452.

\bibitem{Go}
Gerd Gotzmann, \emph{Eine {B}edingung f\"ur die {F}lachheit und das
  {H}ilbertpolynom eines graduierten {R}inges}, Math. Z. \textbf{158} (1978),
  no.~1, 61--70.

\bibitem{Gr}
Mark~L. Green, \emph{Generic initial ideals}, Six lectures on commutative
  algebra, Mod. Birkh\"auser Class., Birkh\"auser Verlag, Basel, 2010,
  pp.~119--186.

\bibitem{H66}
Robin Hartshorne, \emph{Connectedness of the {H}ilbert scheme}, Inst. Hautes
  \'Etudes Sci. Publ. Math. (1966), no.~29, 5--48.

\bibitem{HoaLecture}
Le~Tuan Hoa, \emph{{C}astelnuovo-{M}umford regularity}, 2018, Available at
  http://math.ipm.ac.ir/commalg/Lectures/18Tehran\_Lect.pdf.

\bibitem{Hulett}
Heather~A. Hulett, \emph{Maximum {B}etti numbers of homogeneous ideals with a
  given {H}ilbert function}, Comm. Algebra \textbf{21} (1993), no.~7,
  2335--2350.

\bibitem{Ma}
F.~S. Macaulay, \emph{Some properties of enumeration in the theory of modular
  systems}, Proc. London Math. Soc. (1926), no.~26, 531--555.

\bibitem{Mall1}
Daniel Mall, \emph{Betti numbers, {C}astelnuovo {M}umford regularity, and
  generalisations of {M}acaulay's theorem}, Comm. Algebra \textbf{25} (1997),
  no.~12, 3841--3852.

\bibitem{Mall2}
\bysame, \emph{Connectedness of {H}ilbert function strata and other
  connectedness results}, J. Pure Appl. Algebra \textbf{150} (2000), no.~2,
  175--205.

\bibitem{Pardue97}
K.~Pardue, \emph{Deformations of graded modules and connected loci on the
  {H}ilbert scheme}, The {C}urves {S}eminar at {Q}ueen's, {V}ol. {XI}
  ({K}ingston, {ON}, 1997), Queen's Papers in Pure and Appl. Math., vol. 105,
  Queen's Univ., Kingston, ON, 1997, pp.~131--149.

\bibitem{Ro}
Lorenzo Robbiano, \emph{Introduction to the theory of {H}ilbert functions}, The
  {C}urves {S}eminar at {Q}ueen's, {V}ol.\ {VII} ({K}ingston, {ON}, 1990),
  Queen's Papers in Pure and Appl. Math., vol.~85, Queen's Univ., Kingston, ON,
  1990, pp.~Exp.\ No.\ B, 26.

\bibitem{R}
Leslie~G. Roberts, \emph{Hilbert polynomials and minimum {H}ilbert functions},
  The curves seminar at {Q}ueens, {V}ol. {II} ({K}ingston, {O}nt., 1981/1982),
  Queen's Papers in Pure and Appl. Math., vol.~61, Queen's Univ., Kingston, ON,
  1982, pp.~Exp. No. F, 21.

\bibitem{RTV2003}
Maria~Evelina Rossi, Ng\^{o}~Vi\^{e}t Trung, and Giuseppe Valla,
  \emph{Castelnuovo-{M}umford regularity and extended degree}, Trans. Amer.
  Math. Soc. \textbf{355} (2003), no.~5, 1773--1786.

\end{thebibliography}

\providecommand{\bysame}{\leavevmode\hbox to3em{\hrulefill}\thinspace}
\providecommand{\MR}{\relax\ifhmode\unskip\space\fi MR }
\providecommand{\MRhref}[2]{%
  \href{http://www.ams.org/mathscinet-getitem?mr=#1}{#2}
}
\providecommand{\href}[2]{#2}

\end{document}